\documentclass[12pt]{article}
\usepackage{amsmath}
\usepackage{amsfonts}
\usepackage{amssymb}
\usepackage{setspace}
\usepackage{amsthm}

\usepackage{mathrsfs}
\usepackage{amsxtra}
\usepackage[all,cmtip]{xy}
\usepackage{graphicx}
\usepackage{color}
\usepackage{comment}
\usepackage{verbatim}
\usepackage{url}
\usepackage{hyperref}

\newtheorem{thm}{Theorem}[section]

\newtheorem{lem}[thm]{Lemma}
\newtheorem{prop}[thm]{Proposition}

\newtheorem{cor}[thm]{Corollary}

\theoremstyle{definition}
\newtheorem{defn}[thm]{Definition}
\newtheorem{rem}[thm]{Remark}
\def\XXint#1#2#3{{\setbox0=\hbox{$#1{#2#3}{\int}$}
\vcenter{\hbox{$#2#3$}}\kern-.5\wd0}}

\def\bA{\mathbb{A}}
\def\cA{{\mathcal A}}

\def\cC{{\mathcal C}}

\def\cK{{\mathcal K}}

\def\cS{{\mathcal S}}

\def\cP{\mathcal{P}}

\def\cH{\mathcal{H}}
\def\cL{{\mathcal L}}
\def\cM{{\mathcal M}}

\def\Z{{\Bbb Z}}
\def\R{{\Bbb R}}

\def\C{{\Bbb C}}

\def\Q{{\Bbb Q}}

\def\Hom{\text{Hom}}

\def\det{\text{det}\,}

\def\GL{\operatorname{GL}}
\def\SL{\operatorname{SL}}

\def\GL{\operatorname{GL}}

\def\sign{\operatorname{sign}}

\def\cP{\mathcal{P}}
\def\Norm{\operatorname{N}}

\def\Func{\operatorname{Func}}
\def\cN{\mathcal{N}}

\begin{document}
\hoffset=-.4in
\voffset=-0.05in

%\small

 \title{Shintani cocycles and $p$-adic measures}
\author{G. Ander Steele}
\maketitle

\begin{abstract}
We apply the constructions of \cite{Ste12} to the Shintani cocycle of Charollois-Dasgupta-Greenberg, obtaining cocycles on arithmetic subgroups of $\GL_n(\Q)$ valued in maps from ``deformation vectors" $\R^n\backslash\Q^n$ to $p$-adic measures. 
\end{abstract}
%\tableofcontents
\section{Introduction}
In \cite{Ste12}, we used Hill's \emph{Shintani cocycle} to construct cocycles $n-1$ cocycles on arithmetic subgroups of $\GL_n(\Q)$, valued in a space of $p$-adic \emph{pseudo-measures}, for any natural number $n$. The main result of \cite{Ste12} is a condition guaranteeing that these cocycles specialize to measures on \emph{non-degenerate} arguments: in other words, the cocycle is measure-valued on inputs with linearly independent first columns.  In this note, we use the Shintani cocycle of Charollois, Dasgupta, and Greenberg to construct cocycles for arithmetic subgroups which always evaluate to $p$-adic \emph{measures}, removing the hypothesis that the argument is non-degenerate at the cost of a more complicated module of values.

\section{Notation and Definitions}
Fix $V$ an $n$-dimensional $\Q$-vector space, and fix a choice of basis $\{b_1,\ldots,b_n\}$.  Let $L_0=\Z b_1+\cdots \Z b_n$.

For each prime $p$, we will write $V_p:=V\otimes_{\Q}\Q_p$, and $L_p:=L\otimes_{\Z}\Z_p$. We write $V_\R:=V\otimes_\Q \R$. Denote by $(V_\R)_+$ the positive orthant $\R_+b_1+\cdots \R_+ b_n$, where $\R_+=(0,\infty)$. 

\subsection{Test functions}
The group of test functions on $V$, denoted $\cS(V)$, is the $\Z$-module of Bruhat-Schwartz functions on the finite adeles $\bA_V^{(\infty)}$. Denote by $\cS(V_p)$ the space of step functions (locally constant, compact support) $f_p:V_p\rightarrow \Z$. For example, if $U\subset V_p$ is a compact open, we will write $[U]$ for the characteristic function of $U$. The group of test functions on $V$ is defined by $\cS(V)=\bigotimes_{p}{}' \cS(V_p)$, where the restricted product means $f_p=[L_p]$ almost everywhere. For each prime $p$, we will denote by $\cS^{(p)}(V)$ the space $\bigotimes_{\ell\neq p}{}' \cS(V_\ell),$ and refer to elements $f'\in\cS^{(p)}(V)$ as a test functions \emph{away from $p$}

Given a lattice $L\subset V$, we define $\cS(L):=\{ f\in\cS(V) : f(v)=0 \text{ for al } v\not\in L\}$. We write $\cS^{(p)}(L)\subset \cS^{(p)}(V)$ for the subgroup of test functions $f'$ away from $p$ such that $f'\otimes[L_p]\in\cS(L)$.

Our convention will be to let $\GL(V)$ act on $V$ on the left. Dually, $\cS(V)$ and $\cS^{(p)}(V)$ are endowed with a \emph{right} $\GL(V)$ action, $(f|\gamma)(v):=f(\gamma v)$. Similarly, we let $\GL(L)$ act on $\cS(L)$ and $\cS^{(p)}(L)$ on the right.

\subsection{Cone functions}
If $v_1,\ldots,v_r$, $r\leq n$, are linearly independent vectors in $V_\R$, we write $C^o(v_1,\ldots,v_r)$ for the set of all positive linear combinations $C^o(v_1,\ldots,v_r)=\{\sum a_i v_i | a_i\in\R_+\}$. $C(v_1,\ldots,v_r)$ will denote the closed cone $C(v_1,\ldots,v_r)=\{\sum a_i v_i | a_i\in\R_+\}$. In either case, we will call the rays in the directions $v_1,\ldots,v_r$ the \emph{extremal rays}. More generally, a \emph{simplicial} cone $C$ is a finite union of open cones (glued along boundaries). A \emph{pointed} cone is a cone that does not contain any lines. A \emph{rational} simplicial cone is a union of pointed open cones generated by rational vectors.

Following Hill \cite{Hi07}, we write $\cK^o_V$ for the abelian group of functions $V_\R-\{0\}\rightarrow \Z$ generated by the characteristic function of rational open cones. We write $\cK_V$ for the group of functions $V_\R\rightarrow\Z$ whose restrictions to $V_\R-\{0\}$ are in $\cK^o_V$. The group $\GL(V)$ acts on $\cK_V$ by
\begin{equation}\label{coneaction}
	(\gamma\cdot \kappa)(v)=\sign(\det\gamma) \kappa(\gamma^{-1} v).
\end{equation}
%If $\kappa_1,\kappa_2$ are cone functions, then we will say $\kappa_1\leq \kappa_2$ if the support of $\kappa_1$ is contained in the support of $\kappa_2$.

%The constant functions $V_\R\backslash \{0\}\rightarrow\Z$ form a submodule of $\cK_V$, and we write $\cL_V$ for the quotient $\cK_V/\Z$.
For example, if $v_1,\ldots,v_n$ are linearly independent vectors of $V$, the rational open cone $C^o(v_1,\ldots,v_n)$ is the set $\{ \sum_{i=1}^n \alpha_i v_i : \alpha_i\in \R_+\}$. Then the characteristic function of this open cone, denoted $[C^o(v_1,\ldots,v_n)]$, is an element of $\cK_V$.

Adopting the convention of \cite{CDG}, we call a set of the form 
\begin{equation*}
	L=\R v_1+\R_+ v_2 +\cdots+\R_+ v_n
\end{equation*}
a \emph{wedge}, where $v_1,\ldots,v_n$ are linearly independent. Define $\cL_V\subset \cK_V$ to be the subgroup generated by the characteristic functions $[L]$ of rational wedges.

\subsection{The Solomon-Hu pairing and pseudo-measures}
In this section, we give an improved exposition of \S4.3 of \cite{Ste12} Fix a lattice $L\subset V$, and for each prime $p$ write $\cM(L_p)=\Hom_{cts}(\cC(L_p,\Q_p),\Q_p)$  for the $\Q_p$-valued measures on $L_p$, i.e. the continuous linear functionals on the space of ($p$-adically) continuous functions from $L_p$ to $\Q_p$. For each $v\in L$, write $\delta_v$ for the Dirac measure $\delta_v(f):=f(v)$. The $\Q_p$-vector space $\cM(L_p)$ is a Banach space under the sup-norm and is a ring under convolution-- the Amice transform identifies $\cM(L_p)$ with the power series ring $\Z_p[[T_1,\ldots,T_n]]\otimes_{\Z_p}\Q_p$. Breaking from the conventions in the literature, we define the space of \emph{pseudo-measures} $\widetilde{\cM}(L_p)$ to be the localization of $\cM(L_p)$ with respect to the multiplicative set generated by $\{1-\delta_v: v\neq 0\}$.

The goal of this section is to construct a bilinear pairing between cone functions
\begin{equation}
	\cS^{(p)}(L)\times\cK_V \longrightarrow \widetilde{\cM}(L_p).
\end{equation}
The construction, a modified version of the Solomon-Hu pairing \cite{SolomonHu}, sends a test function $f$ (supported on $L$) and a cone function $\kappa$, to generating function of the Shintani zeta function $\zeta_{SH}(f,\kappa;s)=\sum_{v\in V^+}f(v)\kappa(v)\Norm(v)^{-s}$.

\begin{defn}
If $C$ is an open pointed-cone, we define the ``$C$-upper half-plane in $V_\C^*:=\Hom_{\C}(V_\C,\C)$ as
\begin{equation}
	\cH_C :=\{ \lambda\in V_\C^*: \lambda(v)\in\cH \text{ for all } v\in\C\}
\end{equation}
\end{defn}

The group of cone functions $\cK_V$ is generated by the characteristic functions of ``open" cones $C^o(v_1,\ldots,v_r)$, $v_1,\ldots,v_r$ linearly independent. Scaling the vectors $v_1,\ldots,v_r$ by positive rational numbers leaves the cone unchanged. If $f$ is a test function, then some positive multiples of the vectors $v_1,\ldots,v_r$ are in fact periods for $f$. Therefore, for each test function $f$, it suffices to consider open cones generated by period vectors $v_1,\ldots,v_r\in L$. Let us fix a test function $f$ and such a cone $C=C^o(v_1,\ldots,v_r)$.f

For each $v\in V$, define $q^v: V_\C^*\longrightarrow \C^\times$ by $q^v(\lambda):=e^{2\pi i \lambda(v)}$. 
\begin{lem}
For all $\lambda\in \cH_C$, the infinite sum
\begin{equation}
	G_{C,f}(\lambda):= \sum_{v\in C} f(v)q^v(\lambda)
\end{equation} 
converges in $\C$ to
\begin{equation}\label{E:rational}
	G_{C,f}(\lambda) = \frac{1}{1-q^{v_1}(\lambda)}\cdots\frac{1}{1-q^{v_r}(\lambda)} \sum_{v\in\cP} f(v)q^{v}(\lambda),
\end{equation}
where $\cP=\{ \sum_{i=1}^r x_i v_i : x_i\in (0,1]\cap\Q\}$. (N.B. the sum is finite, since $f$ is supported on $L$)
\end{lem}
\begin{proof}
If $\lambda\in\cH_C$, then for all $v\in C$ $|q^v(\lambda)|<1$. Expanding the right-hand side of (\ref{E:rational}) as a product of geometric series gives the result.
\end{proof}
Therefore, $G_{C,f}$ defines an analytic function on $\cH_C$ which extends to a meromorphic function on $V_\C^*$. 
By identifying $q^v$ with the Dirac distribution $\delta_v$, we will interpret $G_{C,f}$ as a pseudo-measure on $L_p$, for each prime $p$. To make this precise, consider the ring $\Q[q^v]_{v\in L}$. It is isomorphic to the polynomial ring $\Q[q^{w_1},\ldots, q^{w_n}]$, where $w_1,\ldots,w_n$ is  a basis of $L$.  For each prime $p$, we define a homomorphism $\varphi_{L,p} :\Q[q^v]_{v\in L}\longrightarrow \cM(L_p)$ by sending $\varphi_{L,p}(q^v)=\delta_v$. Localizing $\Q[q^v]_{v\in L}$ with respect to the multiplicative set generated by $\{1-q^v: 0\neq v\in L\}$ induces a ring homomorphism 
\begin{equation}
	\varphi_{L,p} :(\Q[q^v]_{v\in L})_{(1-q^v)}\longrightarrow \widetilde{\cM}(L_p).
\end{equation}

We state without proof the following version of the Solomon-Hu pairing \cite{SolomonHu}; see also Hill's exposition \cite{Hi07}.
\begin{prop}[\cite{SolomonHu}]
The map $(C,f)\mapsto \varphi_{L,p}G_{C,f}$ induces a bilinear homomorphisms $\langle ~,~\rangle: \cK_V/\cL_V\times\cS(L)\longrightarrow \widetilde\cM(L_p)/\Z\delta_0$. \end{prop}

\begin{rem}
By taking the product of test functions away from $p$ with $[L_p$, we get a homomorphism $\cS^{(p)}(L)\xrightarrow{\otimes[L_p]}\cS(L)$. This induces a bilinear homomorphism $\cK_V/\cL_V\times\cS^{(p)}(L)\longrightarrow\widetilde{\cM(L_p)}/\Z\delta_0$.
\end{rem}

We record two easy lemmas about this Solomon-Hu pairing.
\begin{lem}\label{L:G-action}
For all $\gamma\in \GL^+(L)$, $\langle \gamma\cdot \kappa, f|\gamma^{-1}\rangle = \gamma \langle \kappa,f\rangle$.
\end{lem}
\begin{proof}
From the $q$-expansions, we have
\begin{align}
	\langle \kappa|\gamma,f|\gamma\rangle = \sum_{v\in V} \kappa(\gamma^{-1} v) f(\gamma^{-1} v) q^v\\
	=\sum_{w\in V} \kappa(w) f(w) q^{\gamma w}\\
	=\gamma \cdot\sum_{w\in V}\kappa(w)f(w) q^w\\
	=\gamma\cdot \langle \kappa,f\rangle.
\end{align}
\end{proof}
\begin{lem} 
If $\kappa=[C^o(v_1,\ldots,v_r)]$ with $v_1\ldots,v_r\in L$ periods of $f$, then
\begin{equation}\label{E:pseudo-measure}
	\langle \kappa, f\rangle  = \frac{1}{1-\delta_{v_1}}\cdots\frac{1}{1-\delta_{v_r}} \sum_{v\in \cP} f(v) \delta_v
\end{equation}
\end{lem}
\begin{proof}
This follows directly from the definition.
\end{proof}

\begin{rem}
It should come as no surprise that if $\langle \kappa, f'\rangle $ is actually a measure, then it has as moments special values of Shintani zeta functions. See, for example, \cite{Ste12}, Proposition 4.10. Finding conditions for these pseudo-measures to be measures is the main result of \cite{Ste12}; we state the result in the next section and present simplified proofs in the appendix.
\end{rem}

\subsection{Measure criteria}
Given a test function $f\in\cS(V_\ell)$, a non-zero $v\in V$ and any $w\in V$, write $f_{v,w}\in \cS(\Q_\ell)$ for the (a fortiori) test function $f_{v,w}:\Q_\ell\rightarrow \Z$ 
\begin{equation}
	(f_{v,w})(t):=f(w+v t ), \text{ for all } t\in\Q_\ell.
\end{equation}
Similarly, we define $f'_{v,w}\in \cS(\Q^{(p)})$ and $f_{v,w}:\in\cS(\Q)$.

\begin{defn}[{\bf Vanishing Hypothesis}]
Let $v\in V$ be a non-zero vector. We will say a test function $f'$ \emph{satisfies the {Vanishing Hypothesis} for $v$} if $h^{(p)}(f_{v,w}' )=0$ for all $w\in V$.
\end{defn}

We remark that $h^{(p)}(f_{v,w}')$ depends only on $w \mod \langle v\rangle$ by the translation invariance of Haar. 
\begin{thm}\label{T:key_technical}
Let $C$ be a simplicial cone with extremal rays $v_1,\ldots,v_r$. The pseudo-measure $\mu_{f',C}\in\widetilde{\cM}(L_p)$ is a measure if $f'$ satisfies the vanishing hypothesis for $v_1,\ldots,v_r$.
\end{thm}

\section{Shintani cocycles}
\subsection{The Charollois-Dasgupta-Greenberg cocycle}
We briefly describe the Shintani cocycle of Charollois-Dasgupta-Greenberg.  Fix a basis of $V\cong \Q^n$ and an auxiliary vector $Q\in V_{irr}:=\R^n-\Q^n$. For vectors $v_1,\ldots,v_n$, the $Q$-deformed cone $c_Q(v_1,\ldots,v_n)\in\cK_V$ is the cone function
\begin{equation}
	c_Q(v_1,\ldots,v_n) (w):=\begin{cases} \lim_{\varepsilon\rightarrow 0^+} [C^o(v_1,\ldots,v_n)](w+\varepsilon Q) & \text{ if $v_1,\ldots, v_n$ are linearly independent,}\\ 0 & \text{ otherwise.}\end{cases}
\end{equation}
For any $\GL(V)$-module $A$, let $\mathcal{N}:=\Func(V_{irr},A)$ be the module of maps from ``deformation vectors" $Q$ to $A$. It is a $\GL_n(\Q)$-module under the action $(\gamma\cdot F)(Q)=\gamma\cdot F(\gamma^{-1}Q)$. The Charollois-Dasgupta-Greenberg cocycle $\Psi_{CDG}:\GL_n(\Q)^n\longrightarrow\cN(\cK_V/\cL_v)$ is defined by 
\begin{equation}
	\Psi_{CDG}(\alpha_1,\ldots,\alpha_n)(Q) =\sign(\det(\alpha_1 e_1 \ldots \alpha_n e_1)) c_Q(\alpha_1e_1,\ldots,\alpha_n e_1).
\end{equation}

\begin{thm}[Theorem 1.6 of \cite{CDG}]
The map $\Psi_{CDG}$ defines an $n-1$ cocycle
\begin{equation}
\Psi_{CDG}\in Z^{n-1}(\GL_n(\Q),\mathcal{N}(\cK_V/\cL_V)).
\end{equation}
\end{thm}

For our purposes, Charollois-Dasgupta-Greenberg's cocycle has a definite advantage over Hill's cocycle: the support of the cone functions can be described precisely, at the cost of introducing the more complicated module $\cN$. We remark that the difference in modules makes comparing these cocycles difficult. Interestingly, the Charollois-Dasgupta-Greenberg cocycle is trivial on the mirabolic subgroup $\{ \alpha\in\GL_n(\Q) : \alpha e_1=e_1\}$, which is not the case with Hill's cocycle. 
\begin{prop}\label{P:CDG}
Let $\alpha_1,\ldots,\alpha_n\in\GL_n(\Q)$. If $\alpha_1e_1,\ldots,\alpha_ne_1$ are linearly independent, then, for all $Q\in V_{irr}$, $\Psi_{CDG}(\alpha_1,\ldots,\alpha_n)(Q)$ is the sum of characteristic functions of open cones generated by subsets of $\{\alpha_1e_1,\ldots,\alpha_ne_1\}$. Otherwise, $\Psi_{CDG}(\alpha_1,\ldots,\alpha_n)(Q)=0$.
\end{prop}
\begin{proof}
Equation (7) of \cite{CDG} describes which faces of $c_Q(\alpha_1e_1,\ldots,\alpha_n e_1)$ are included.
\end{proof}
\subsection{The $\Gamma_{f'}$-cocycle}
Fixing a test function $f'$, we let $\Gamma_{f'}\subset\SL_n(\Z)$ denote the stabilizer of $f'$. Note that $\Gamma_{f'}$ is an arithmetic subgroup, since it is the stabilizer of the test function $f'\otimes[L_p]$, and thus is commensurable with $\SL_n(\Z)$.
\begin{defn}
The improved $p$-adic Shintani cocycle for $\Gamma_{f'}$ is the composition
\begin{equation}
	\Phi_{f'}:\Gamma_{f'}^n\xrightarrow{\Psi_{CDG}}\cN(\cK_V/\cL_V)\xrightarrow{\langle ~, f'\otimes[L_p]\rangle}\cN(\widetilde{\cM}(\Z_p^n)/\Z\delta_0).
\end{equation}
By Lemma \ref{L:G-action}, $\Phi_{f'}$ is $\Gamma_{f'}$-equivariant. It satisfies the cocycle condition by Theorem 3.1 and is thus $n-1$ cocycle for $\Gamma_{f'}$ valued in $\cN(\widetilde\cM(\Z_p^n)/\Z\delta_0)$.
\end{defn}
\begin{thm}
If $f'$ satisfies the vanishing hypothesis for $e_1$, then $\Phi_{f'}(Q)$ is valued in the submodule of $p$-adic measures $\cM(\Z_p^n)/\Z\delta_0 \subset \widetilde{\cM}(\Z_p^n)/\Z\delta_0$ for all $Q\in V_{irr}$. In other words,
\begin{equation*}
	\Phi_{f'}\in Z^{n-1}(\Gamma_{f'},\cN(\cM(\Z_p^n)/\Z\delta_0)).
\end{equation*}	
\end{thm}
\begin{proof}
If $f'$ satisfies the vanishing hypothesis for $e_1$, then it satisfies vanishing hypothesis  for $\gamma e_1$ for all $\gamma\in \Gamma_{f'}$. Proposition \ref{P:CDG} tells us that $\Psi_{CDG}(\alpha_1,\ldots,\alpha_n)(Q)$ is the sum of characteristic functions of open cones generated by rays in $\{\alpha_1,\ldots,\alpha_n\}$, and in particular rays for which $f'$ satisfies VH. By Theorem \ref{T:key_technical} $\Phi_{f'}(\alpha_1,\ldots,\alpha_n)(Q)=\langle\Psi_{CDG}(\alpha_1,\ldots,\alpha_n)(Q),f'\rangle$ is the sum of $p$-adic measures.
\end{proof}

\section{Appendix}
We prove Theorem 2.7 in several steps, first interpreting $h(f_{v,w})$ in terms of $q$-expansions of $G_{C,f}$. Moving to $p$-adic pseudo-measures, the vanishing hypothesis guarantees that the pseudo-measures don't have poles.
\begin{prop}\label{P:q-series}
Let $f$ be a test function, and suppose $C=C^o(v_1,\ldots,v_r)$ is an open cone. For each extremal ray $v_i$ and $w\in C^o(v_1,\ldots,\widehat{v_i},\ldots,v_r)$, 
\begin{equation}
	\text{The coefficient of $q^w$ in } (1-q^{v_i})G_{C,f}|_{q^v=1}=  h(f_{v_i,w})
\end{equation}
\end{prop}
\begin{proof}
For notational convenience, we assume without loss of generality that $i=1$. Taking the $q$-expansion of $G_{C,f}$ gives
\begin{align}
	(1-q^{v_1})G_{C,f} = (1-q^{v_1})\sum_{v\in C} f(v)q^v\\
	=(1-q^{v_1})\sum_{x_1,\ldots,x_n>0}f(x_1v_1+\cdots+x_rv_r) q^{x_1v_1+\cdots x_rv_r}\\
	=\sum_{x_1,\ldots,x_n>0}f(x_1v_1+\cdots+x_rv_r) q^{x_1v_1+\cdots x_rv_r}-\sum_{x_1,\ldots,x_n>0}f(x_1v_1+\cdots+x_rv_r) q^{(x_1+1)v_1+\cdots x_rv_r}\\
\end{align}
Since $f$ is periodic with respect to $v_1$, $f(x_1v_1+\cdots x_nv_n)=f((x_1+1)v_1+\cdots+x_nv_n)$, and the differences of sums telescopes to
\begin{equation}
\sum_{x_2,\ldots,x_n>0} \sum_{x_1\in (0,1]\cap\Q} f(x_1v_1+\cdots +x_r v_r)q^{x_1v_1}q^{x_2v_2+\cdots+x_rv_r}.
\end{equation}
Evaluating at $q^{v_1}=1$ (i.e. restricting the meromorphic function $G_{C,f}$ to linear functionals vanishing on $v_1$), we get
\begin{align}
	(1-q^{v_1})G_{C,f}|_{q^{v_1}=1}=\sum_{x_2,\ldots,x_n} \sum_{x_1\in(0,1]\cap \Q} f(x_1v_1 +(x_2v_2+\cdots+x_nv_n)) q^{x_2v_2+\cdots+x_nv_n}\\
	=\sum_{w\in C^o(v_2,\ldots,v_n)} \left(\sum_{x_1\in (0,1]\cap\Q} f(x_1v_1 +w)\right)q^w\\
	=\sum_{w\in C^o(v_2,\ldots,v_n)}\left(\sum_{x_1\in (0,1]\cap\Q} f_{v_1,w}(x_1)\right)q^w
\end{align}
Since $v_1$ is a period of $f$, $f_{v_1,w}\in\cS(\Q)$ has $1$ as a period. The term in the parenthesis is exactly the Haar measure of $f_{v_1,w}$ (normalized with respect to $\Z)$. Indeed, if $P$ is a period lattice of $f\in\cS(V)$, then 
\begin{equation}
	h:(f) = \frac{1}{[L:P]}\sum_{v\in V/L} f(v).
\end{equation}
See, for example, \cite{Ste89}. Taking $L=\Z$, $(0,1]\cap\Q$ is a set of representatives of $\Q/L$, giving
\begin{equation}
	(1-q^{v_1})G_{C,f}|_{q^{v_1}=1}=\sum_{w\in C^o(v_2,\ldots,v_n)} h(f_{v_1,w})q^w
\end{equation}
and hence the result.
\end{proof}

\begin{prop}\label{P:easy-measure}
Let $v_1,\ldots,v_n\in L$ be linearly independent, and put $U_p=\Z_pv_1+\cdots \Z_p v_n\subset L_p$ and $\kappa=[C^o(v_1,\ldots,v_r)]$ for some $1\leq r\leq n$. The pseudo-measure $\mu_{\kappa,f'\otimes[U_p]}=\langle\kappa,f'\otimes[U_p]\rangle \in \widetilde{\cM}(L_p)$ is a measure if and only if $f'$ satisfies the vanishing hypothesis for $v_1,\ldots,v_r$. 
\end{prop}
\begin{proof}
Restriction of continuous functions from $L_p$ to $U_p$ induces an inclusion $\cM(U_p)\hookrightarrow\cM(L_p)$. Writing $\widetilde\cM(U_p)$ for the localization of $\cM(U_p)$ with respect to $\langle 1-\delta_v : v\in U_p, v\neq 0\rangle$, we have an inclusion $\widetilde\cM(U_p)\hookrightarrow\widetilde\cM(L_p)$. We first note that $\mu:=\mu_{\kappa,f'\otimes[U_p]}$ is an element of the image of $\widetilde{\cM}(U_p)$: Since $[U_p]$ is periodic with respect to $v_1,\ldots,v_r$, there exist $p$-adic units $\alpha_1\ldots,\alpha_r\in \Z_p^\times$ such that $f:=f'\otimes[U_p]$ is periodic with respect to $\alpha_1 v_1\ldots,\alpha_r v_r$. Therefore, 
\begin{equation}
	\mu_{\kappa,f} = \frac{1}{1-\delta_{\alpha_1v_1}}\cdots\frac{1}{1-\delta_{\alpha_rv_r}}\sum_{v\in\cP} f(v)\delta_v,
\end{equation}
where $\cP:=\{ \sum_{i=1}^r x_i v_i : x_i\in (0,\alpha_i]\cap\Q\}.$ Since $f$ vanishes at $v\not\in U_p$, each term belongs to $\widetilde{\cM}(U_p)$. Therefore it suffices to show that $\mu_{\kappa,f}$, viewed as an element of $\widetilde{\cM}(U)$, is a measure exactly when $f'$ satisfies the vanishing hypothesis for $v_1,\ldots,v_r$.

%{\bf Claim:} $(1-\delta_{\alpha_iv_i})$ divides $\sum_{v\in \cP} f'\otimes[U_p](v) \delta_v$ in $\cM(U_p)$ if and only if $f'$ satisfies the vanishing hypothesis for $v_i$.
The basis $v_1,\ldots,v_n$ identifies $U_p$ with $\Z_p^n$ and the Amice transform $\cA:\cM(U_p)\xrightarrow{~\sim~}\Z_p[[T_1,\ldots,T_n]]\otimes_{\Z_p}\Q_p$ is defined by
\begin{align}
	\cA(\mu)=
	\sum_{k_1,\ldots,k_n\geq 0} \left(\int_{\Z_p^n} {x_1 \choose k_1}\cdots {x_n\choose k_n} d\mu(x_1v_1+\cdots x_nv_n)\right)T_1^{k_1}\cdots T_n^{k_n}\\
	=\int_{\Z_p^n} (1+T_1)^{x_1}\cdots(1+T_n)^{x_n}d\mu(x_1v_1+\cdots x_n v_n).
\end{align}
Under this identification $\cA(\delta_{v_i})=1+T_i$ and a standard computation shows $\cA(\delta_{\alpha v_i})=(1+T_i)^{\alpha}$ for each $\alpha\in\Z_p$.   Therefore, $\mu_{\kappa,f'\otimes[U_p]}$ is a measure if and only if $(1+T_i)^{\alpha_i}-1$ divides $\sum_{v\in \cP} f(v)\cA(\delta_v)$. Now $\frac{(1+T_i)^{\alpha_i}-1}{T}$ is a unit in $\Z_p[[T_1,\ldots,T_n]]$, so the ideal $((1+T_i)^{\alpha_i}-1)$ is equal to the ideal $(T_i)$. A power series $F(T_1,\ldots,T_n)$ belongs to this ideal if and only if $F(T_1,\ldots,T_n)|_{T_i=0}=0$. To apply this to our pseudo-measures, observe that evaluating the Amice transform of $(1-\delta_{\alpha_iv_i})\mu_{\kappa,f'\otimes[U_p]}$ at $T_i=0$ corresponds to restricting the meromorphic function $(1-q^{v_i})G_{C,f}$ to linear functionals vanishing at $v_i$. Proposition \ref{P:q-series} tell us the coefficients of the $q$-expansion of $(1-q^{v_i})G_{C,f}|_{q^{v_i}=1}$ are given by $h(f_{v_i,w})=h^{(p)}(f'_{v_i,w})h_p([\Z_p])$. It follows that $\cA((1-\delta_{\alpha_iv_i})\mu_{\kappa,f})|_{T_i=0}$  vanishes exactly when $f'$ satisfies the vanishing hypothesis for $v_i$. We conclude $\mu_{\kappa,f'\otimes[U_p]}$ is a measure if and only if $f'$ satisfies the vanishing hypothesis for $v_1,\ldots,v_r$.
\end{proof}

\begin{cor}
With the same notation as above, suppose furthermore that $u\in L_p$. Then $\mu_{\kappa,f'\otimes[u+U_p]}=\langle \kappa,f'\otimes[u+U_p]\rangle$ is a measure if and only if $f'$ satisfies the vanishing hypothesis for $v_1,\ldots,v_r$.
\end{cor}
\begin{proof}
We have
\begin{align}
	\langle \kappa, f'\otimes[u+U_p]\rangle =\frac{1}{1-\delta_{v_1}}\cdots\frac{1}{1-\delta_{v_r}} \sum_{v\in\cP} f'(v)[u+U_p](v) \delta_v\\
	=\frac{1}{1-\delta_{v_1}}\cdots\frac{1}{1-\delta_{v_r}} \sum_{v\in\cP} f'(v)[U_p](v-u) \delta_v.
\end{align}
Write $g'$ for the test function $g'(v-u)=f'(v)$. Note that, by the translation invariance of Haar, $g'$ satisfies the vanishing hypothesis for $v_1\ldots,v_r$ if and only if $f'$ does. Now $\langle \kappa,f'\otimes[u+U_p]\rangle$ becomes
\begin{align}
	\frac{1}{1-\delta_{v_1}}\cdots\frac{1}{1-\delta_{v_r}} \sum_{v\in\cP} g'(v-u)[U_p](v-u) \delta_v\\
	=\frac{1}{1-\delta_{v_1}}\cdots\frac{1}{1-\delta_{v_r}} \sum_{v\in \cP} g'\otimes[U_p](v-u)\delta_u\delta_{v-u}\\
	=\delta_u\left(\frac{1}{1-\delta_{v_1}}\cdots\frac{1}{1-\delta_{v_r}}\sum_{ w\in u+\cP} g'\otimes[U_p](w)\delta_w\right).
\end{align}
Since $u+\cP$ is a fundamental domain for the action of $U_p$, the proof of Proposition \ref{P:easy-measure} shows the term in the parentheses is a measure if and only if $g'$ (and hence $f'$) satisfies the vanishing hypothesis for $v_1,\ldots,v_r$. Therefore, $\langle \kappa,f'\otimes[u+U_p]\rangle$ is the convolution of a measure (and itself a measure) if and only if $f'$ satisfies the vanishing hypothesis for $v_1,\ldots,v_r$.
\end{proof}

Now we turn to the proof of the main theorem
\begin{proof}[Proof of Theorem 2.7]
By scaling $v_1,\ldots,v_r$ by positive rationals, we may assume without loss of generality that $v_1,\ldots,v_r\in L$. Extending, if necessary, $v_1,\ldots,v_r$ to a basis of $V$, we again put $U_p=\Z_pv_1+\cdots \Z_p v_n\subset L_p$. The lattice $U_p$ has finite index in $L_p$, so there exists a finite number of vectors $\{u_i\}$ such that $[L_p]=\sum_{i=1}^{[L_p:U_p]}[u_i+U_p]$. It follows that $\mu_{\kappa,f'}=\sum_{i=1}^{[L_p:U_p]}\mu_{\kappa,f'\otimes[u_i+U_p]}$. By the corollary, each of these pseudo-measures is a measure exactly when $f'$ satisfies the vanishing hypothesis for $v_1,\ldots,v_r$. In this case, $\mu_{\kappa,f'}$ is a measure.
\end{proof}

%\section{Solomon-Hu pairing revisited}
%
%Given a pointed (simplicial) cone $C$, we define the ``$C$-upper half plane" $\cH_C\subset V_\C^\ast$ to be the set of linear functionals $\cH_C:=\{\tau \in V_\C^\ast | \tau(v)\in\cH \text{ for all }v\in C\}$.
%
%If $f\in\cS(V)$ is a test function and $C$ is as above, then we define $\Lambda(f,C):=\{ \lambda\in V_\Q^\ast | \text{ for all} v\in C \text{ s.t. } f(v)\neq 0,~\lambda(v)\in\Z\}$.
%
%Is $V^\ast_\C/\Lambda(F,C)$ a torus? certainly $V^\ast/\Lambda(F,C)$ is a torus
%
%
%\begin{lem}
%For each $w\in V$, the coefficient of $q^w$ in the $q$-expansion of $(1-q^{v_i})\cA(\mu_{f',C,U})|_{q_i=1}$ is $h^{(p)}(f'_{v_i,w})$.
%\end{lem}
%
%\begin{cor}
%If $f'$ satisfies the vanishing hypothesis for $v_i$, $\cA(\mu_{f',C,U})\in (1-q^{v_i})\cA(\cB)$.
%\end{cor}
%
%\begin{prop}
%If $f'$ satisfies the vanishing hypothesis for the extremal rays of $C$, $\mu_{f',C,U}$ is a measure.
%\end{prop}
%
%\begin{thm}
%If $f'$ satisfies the vanishing hypothesis for the extremal rays of $C$, $\mu_{f',C}$ is a measure.
%\end{thm}

%\bibliographystyle{amsplain}
%\bibliography{Shintanibib}
\providecommand{\bysame}{\leavevmode\hbox to3em{\hrulefill}\thinspace}
\providecommand{\MR}{\relax\ifhmode\unskip\space\fi MR }
\providecommand{\MRhref}[2]{%
  \href{http://www.ams.org/mathscinet-getitem?mr=#1}{#2}
}
\providecommand{\href}[2]{#2}

\end{document}